\documentclass[english,12pt]{article}
\usepackage{amsmath,amsfonts,amsthm,amssymb,color,mathrsfs}
\usepackage[T1]{fontenc}
\usepackage[utf8]{inputenc}
\usepackage[left=1in, right=1in, top=1.1in,bottom=1.1in]{geometry}
\usepackage{amsthm}
\usepackage{amsmath}
\usepackage{amssymb}
\usepackage{wasysym}
\usepackage{esint}
\usepackage{hyperref}
\usepackage{comment}

\usepackage[title]{appendix}

\hypersetup{
     colorlinks   = true,
     citecolor    = blue,
     linkcolor    = blue
}
\usepackage{pdfsync}
\usepackage{stmaryrd}

\makeatletter

\newcommand{\bch}{\bar{\mathcal{H}}}

\def\mr{{\mathbb  R}}

\newcommand{\beq}{\begin{equation}}
\newcommand{\eeq}{\end{equation}}

\newcommand{\1}{{\bf 1}}

\newcommand{\R}{\mathbb R}

\newcommand{\cac}{\mathcal C}

\newcommand{\cf}{\mathcal F}
\newcommand{\ch}{\mathcal H}

\newcommand{\ck}{\mathcal K}

\newcommand{\cs}{\mathcal S}

\newcommand{\al}{\alpha}
\newcommand{\ep}{\varepsilon}

\newcommand{\oom}{\Omega}

\newcommand{\vp}{\varphi}

\newcommand{\lp}{\left(}
\newcommand{\rp}{\right)}

\newcommand{\lcl}{\left\{}
\newcommand{\rcl}{\right\}}

\numberwithin{equation}{section}
  \theoremstyle{plain}
  \newtheorem{thm}{\protect\theoremname}[section]
   \theoremstyle{plain}
  
  \theoremstyle{remark}
  \newtheorem{rem}[thm]{\protect\remarkname}
  \theoremstyle{definition}
  \newtheorem*{notation*}{\protect\notationname}
  \theoremstyle{plain}
  \newtheorem{prop}[thm]{\protect\propositionname}
    \theoremstyle{plain}
  \newtheorem{lem}[thm]{\protect\lemmaname}
  \theoremstyle{plain}
  
  \theoremstyle{definition}
  \newtheorem{defn}[thm]{\protect\definitionname}
   \theoremstyle{plain}
  
   \theoremstyle{plain}
  \newtheorem*{uha*}{\protect\uhaname}
     \theoremstyle{definition}
      \newtheorem*{uell*}{\protect\uellname}
     \theoremstyle{definition}

  \makeatother

\usepackage{babel}
  \providecommand{\assumptionname}{Assumption}
  \providecommand{\definitionname}{Definition}
  \providecommand{\lemmaname}{Lemma}
  \providecommand{\notationname}{Notation}
  \providecommand{\propositionname}{Proposition}
  \providecommand{\remarkname}{Remark}
\providecommand{\corollaryname}{Corollary}
\providecommand{\theoremname}{Theorem}
\providecommand{\claimname}{Claim}
\providecommand{\uhaname}{Uniform Hypoellipticity Assumption}
\providecommand{\uellname}{Uniform Ellipticity Assumption}
\providecommand{\convname}{Convention}

\begin{document}

\title{Precise Local Estimates for Differential Equations driven by Fractional Brownian Motion: Elliptic Case}

\author{Xi Geng\thanks{School of Mathematics and Statistics, University of Melbourne, Melbourne, Australia. Email: xi.geng@unimelb.edu.au.} 
\and 
Cheng Ouyang\thanks{Department of Mathematics, Statistics and Computer Science, University of Illinois at Chicago, Chicago, United States. Email: couyang@math.uic.edu.} 
\and 
Samy Tindel\thanks{Department of Mathematics, Purdue University, West Lafayette, United States. Email: stindel@purdue.edu.} 
}
\date{}

\maketitle

\begin{abstract}
This article is concerned with stochastic differential equations driven by a $d$ dimensional fractional Brownian motion with Hurst parameter $H>1/4$, understood in the rough paths sense. Whenever the coefficients of the equation satisfy a uniform ellipticity condition, we establish a sharp local estimate on the associated control distance function and a sharp local lower estimate on the density of the solution. 
\end{abstract}

\tableofcontents

\section{Introduction}
In this paper, we consider the following stochastic differential equation (SDE)
\begin{align}\label{sde-intro}
X_t=x+\sum_{i=1}^d \int_0^t V_i(X_s)dB^i_s,\quad t\in[0,1],
\end{align}
where $x \in \mathbb{R}^N$, $V_1,\cdots,V_d$ are $C^\infty$-bounded  vector fields on $\mr^N$ and $\{B_t\}_{0\leq t\leq 1}$ is an $d$-dimensional  fractional Brownian motion. We assume throughout the paper that in \eqref{sde-intro} the fractional Brownian motion has Hurst parameter $H\in (1/4, 1)$ and that the  vector fields $V_i$'s satisfy the uniform ellipticity condition. When $H\in(1/2,1)$, the above equation is understood in Young's sense \cite{Young36}; and when $H\in(1/4,1/2)$  stochastic integrals in equation~(\ref{sde-intro}) are  interpreted as rough path integrals (see, e.g., ~\cite{FV10, Gu}) which extends the Young's integral. Existence and uniqueness of solutions to the above equation can be found, for example, in \cite{LQ}. In particular, when $H=\frac{1}{2}$, this notion of solution coincides with the solution of the corresponding Stratonovitch stochastic differential equation.  

It is now well understood that under H\"{o}rmander's condition the law of the solution $X_t$ to equation \eqref{sde-intro} admits a smooth probability density $p(t,x,y)$ with respect to the Lebesgue measure on $\mathbb{R}^N$ (cf. \cite{BH07,CF10,H-P,CHLT15}).  Moreover, it is shown in \cite{BNOT16} that, under uniform ellipticity condition,  the following global upper bound holds,
\begin{align}\label{upper bound}p(t,x,y)\leq C\frac{1}{t^{NH}}\exp\left[-\frac{|x-y|^{(2H+1)\wedge 2}}{Ct^{2H}}\right].\end{align}
Clearly \eqref{upper bound} is  of Gaussian type and sharp when $H\geq1/2$; while it  only gives a sub-Gaussian bound  when $H<1/2$. Whether one should still expect a Gaussian upper bound when $H<1/2$ remains one of the major open problems in the study of the density function. Another open problem in this direction is to obtain a sharp lower bound for the density $p(t,x,y)$.

On  the other hand, the Varadhan type estimate established in \cite{BOZ15} shows that
\begin{equation}\label{eq:varadhan-estimate}
\lim_{t\rightarrow0}t^{2H}\log p(t,x,y)=-\frac{1}{2}d(x,y)^{2}.
\end{equation}
In the above, the control distance function $d(x,y)$ is given by
\begin{align}\label{def: distance}
d^2(x,y)=\inf\{ \|h\|_{\bar{\mathcal{H}}}^2; \, \Phi_1(x;h)=y\},
\end{align}
where $\bar{\mathcal{H}}$ is the Cameron-Martin space of $B$ and $\Phi_t(x;\cdot): \bar{\mathcal{H}}\to C[0,1]$ is the deterministic It\^{o} map associated to equation \eqref{sde-intro}.  Although one can not directly equate the Varadhan estimate in \eqref{eq:varadhan-estimate} to the upper bound (or a similar lower bound) in \eqref{upper bound}, it naturally motivates the following questions:
\begin{itemize}
\item[Q1.] Is the control distance $d(x,y)$ comparable to the Euclidean distance $|x-y|$\, ?

\item[Q2.] Can we use techniques developed in proving \eqref{eq:varadhan-estimate} to obtain some information on the bounds of $p(t,x,y)$\,? [Here we are in particular interested in a lower bound, since progress on the lower bound of the density is limited in the literature.]
\end{itemize}

Our investigation in the present article shows an effort in answering the above two questions, at least partially. More specifically, our discovery is reported in the following two theorems.

\begin{thm}\label{thm: local comparison}
Let $d$ be the control distance given in \eqref{def: distance}. Under uniform ellipticity conditions 
{(see the forthcoming equation~\eqref{eq:hyp-elliptic} for a more explicit version),} 
there exist constants $C,\delta>0$, such that 
\begin{equation}\label{eq:local-comparison}
\frac{1}{C}|x-y|\leq d(x,y)\leq C|x-y| \, ,
\end{equation}
for all $x,y\in \mathbb{R}^N$ with $|x-y|<\delta$.
\end{thm}

\begin{rem}
Theorem \eqref{thm: local comparison} reflects our attempt in  answering Q1. The control distance $d(x,y)$ plays an important role in various analytic properties of $X$ in equation \eqref{sde-intro}, for example, the large deviations of $X_t$. Due to the complexity of the Cameron-Martin structure of $B$, the control distance $d(x,y)$ is far from being a metric (for example, it is not clear whether it satisfies the triangle inequality) and its shape is not clear. Our investigation shows that $d(x,y)$, as a function, is locally comparable to the Euclidean distance. The authors believe that a global equivalence would not hold in general.
\end{rem}

Our second result concerns Q2 above and aims at obtaining a lower bound of the density function. 
{It is phrased below in a slightly informal way, and we refer to Theorem~\ref{thm: local lower estimate in elliptic case} for a complete statement.}
\begin{thm}\label{thm: local lower estimate}
Let $p(t,x,y)$ be the probability density of $X_{t}$. Under uniform ellipticity conditions on the vector fields in $V$,
there exist some constants $C,\tau>0$ such that 
\begin{equation}\label{eq:local-density-bound}
p(t,x,y)\geq\frac{C}{t^{NH}},
\end{equation}
for all $(t,x,y)\in(0,1]\times\mathbb{R}^N\times\mathbb{R}^N$ with 
$|x-y|\leq t^H,\text{and}
\ t<\tau .$
\end{thm}

\begin{rem}
Relation \eqref{eq:local-density-bound} presents a local  lower bound, both in time and space, for the density function $p_t(x,y)$ .  It is clearly  sharp by a quick examination of the case when $X_t$ is an $N$-dimensional fractional Brownian motion, i.e. when $N=d$ and $V=\mathrm{Id}$.
\end{rem}

{In order to summarize the methodology we have followed for Theorem~\ref{thm: local comparison} and~\ref{thm: local lower estimate}, we should highlight two main ingredients:\\
(i) Some thorough analytic estimates concerning the Cameron-Martin space related to fractional Brownian motions, which are mostly be useful in order to get proper estimates on the distance $d$ defined by~\eqref{def: distance}.\\ 
(ii) A heavy use of Malliavin calculus, Girsanov's theorem in a fBm context and large deviations techniques are invoked for our local lower bound \eqref{eq:local-density-bound}.\\
Our analysis relies thus heavily on the particular fBm setting. Generalizations to a broader class of Gaussian processes seem to be nontrivial and are left for a subsequent publication.
}

\begin{rem}
As one will see below, our argument for both Theorem \ref{thm: local comparison} and \ref{thm: local lower estimate}   hinges crucially on uniform ellipticity of the vector fields. The hypoelliptic case is substantially harder and requires a completely different approach, which will be studied in a companion paper~\cite{GOT20-hypo}.\end{rem}

\begin{rem}
{For sake of clarity and conciseness, we have restricted most of our analysis to equation \eqref{sde-intro}, that is an equation with no drift. However, we shall give some hints at the end of the paper about how to extend our results to more general contexts.}
\end{rem}

\noindent
\textbf{Organization of the present paper.} In Section \ref{sec:prelim}, we present some basic notions from the analysis of fractional Brownian motion. In particular, we provide substantial detail on the Cameron-Martin space of a fractional Brownian motion. This is needed in order to establish the comparison between control distance and the Euclidean distance and  will also be helpful for later references.  Our main results Theorem \ref{thm: local comparison} and \ref{thm: local lower estimate} will then be proved in Section \ref{sec: ellip.}.


\section{Preliminary results.}\label{sec:prelim}

This section is devoted to some preliminary results on the Cameron-Martin space related to a fractional Brownian motion. We shall also recall some basic facts about rough paths solutions to noisy equations.

\subsection{The Cameron-Martin subspace of fractional Brownian motion.}\label{sec: prel.}

Let us start by recalling the definition of fractional Brownian motion.

\begin{defn}\label{def:fbm}
A $d$-dimensional \textit{fractional Brownian motion} with Hurst parameter $H\in(0,1)$ is an $\mathbb{R}^d$-valued continuous centered Gaussian process $B_t=(B_t^1,\ldots,B_t^d)$ whose covariance structure is given by 
\beq\label{eq:cov-fBm}
\mathbb{E}[B_{s}^{i}B_{t}^{j}]=\frac{1}{2}\left(s^{2H}+t^{2H}-|s-t|^{2H}\right)\delta_{ij}
\triangleq R(s,t) \delta_{ij}.
\eeq
\end{defn}

This process is defined and analyzed in numerous articles (cf. \cite{DU97,Nualart06,PT00} for instance), to which we refer for further details.
In this section, we mostly focus on a proper definition of the Cameron-Martin subspace related to $B$. We also prove two general lemmas about this space which are needed for our analysis of the density $p(t,x,y)$. 
Notice that we will frequently identify a Hilbert space with its dual in the canonical way without  further mentioning. 

In order to introduce the Hilbert spaces which will feature in the sequel, consider a one dimensional fractional Brownian motion $\{B_t:0\leq t\leq 1\}$ with Hurst parameter $H\in(0,1)$.   The discussion here can be easily adapted to the multidimensional setting with arbitrary time horizon $[0,T]$. 
Denote $W$ as the space of continuous paths $w:[0,1]\rightarrow\mathbb{R}^{1}$
with $w_{0}=0.$ Let $\mathbb{P}$ be the probability measure over $W$ under which the coordinate process $B_t(w)=w_t$ becomes a fractional Brownian motion.  Let ${\cal C}_{1}$ be the associated first order Wiener chaos,
i.e. ${\cal C}_{1}\triangleq\overline{\mathrm{Span}\{B_{t}:0\leq t\leq1\}}\ {\rm in}\ L^{2}(W,\mathbb{P})$.

\begin{defn}\label{def:bar-H}
Let $B$ be a one dimensional fractional Brownian motion as defined in \eqref{def:fbm}.
Define $\bar{{\cal H}}$ to be the space of elements $h\in W$ which
can be written as 
\beq\label{eq:def-h-in-CM}
h_{t}=\mathbb{E}[B_{t}Z],\ \ \ 0\leq t\leq1,
\eeq
where $Z\in{\cal C}_{1}.$ We equip $\bar{{\cal H}}$ with an inner product structure given by
\[
\langle h_{1},h_{2}\rangle_{\bar{{\cal H}}}\triangleq\mathbb{E}[Z_{1}Z_{2}],\ \ \ h_{1},h_{2}\in\bar{{\cal H}},
\]
whenever $h^{1},h^{2}$ are defined by \eqref{eq:def-h-in-CM} for two random variables $Z_{1},Z_{2}\in{\cal C}_{1}$.
The Hilbert space $(\bar{\mathcal{H}},\langle\cdot,\cdot\rangle_{\bar{\mathcal{H}}})$ is called the \textit{Cameron-Martin
subspace} of the fractional Brownian motion. 
\end{defn}

One of the advantages of working with fractional Brownian motion is that a convenient analytic description of $\bar{\mathcal{H}}$ in terms of fractional calculus is available (cf. \cite{DU97}).  Namely recall that given a function $f$ defined on $[a,b]$, the right and left \textit{fractional integrals} of $f$ of order $\alpha>0$ are respectively defined by
\beq\label{eq:def-frac-integral}
(I_{a^{+}}^{\alpha}f)(t)\triangleq\frac{1}{\Gamma(\alpha)}\int_{a}^{t}f(s)(t-s)^{\alpha-1}ds,
\quad\text{and}\quad 
(I_{b^{-}}^{\alpha}f)(t)\triangleq\frac{1}{\Gamma(\alpha)}\int_{t}^{b}f(s)(s-t)^{\alpha-1}ds.
\eeq
In the same way the right and left \textit{fractional derivatives} of $f$ of order $\alpha>0$ are respectively defined by  
\beq\label{eq:def-frac-deriv}
(D_{a^{+}}^{\alpha}f)(t)\triangleq\left(\frac{d}{dt}\right)^{[\alpha]+1}(I_{a^{+}}^{1-\{\alpha\}}f)(t),
\quad\text{and}\quad 
(D_{b^{-}}^{\alpha}f)(t)\triangleq\left(-\frac{d}{dt}\right)^{[\alpha]+1}(I_{b^{-}}^{1-\{\alpha\}}f)(t),
\eeq
where $[\alpha]$ is the integer part of $\alpha$ and $\{\alpha\}\triangleq\alpha-[\alpha]$ is the fractional part of $\alpha$. The following formula for $D_{a^+}^\alpha$, valid for $\al\in(0,1)$, will be useful for us:
\begin{equation}\label{eq: formula for fractional derivatives}
(D_{a^{+}}^{\alpha}f)(t)=\frac{1}{\Gamma(1-\alpha)}\left(\frac{f(t)}{(t-a)^{\alpha}}+\alpha\int_{a}^{t}\frac{f(t)-f(s)}{(t-s)^{\alpha+1}}ds\right),\ \ \ t\in[a,b].
\end{equation}
The fractional integral and derivative operators are inverse to each other. For this and other properties of fractional derivatives, the reader is referred to \cite{KMS93}.

Let us now go back to the construction of the Cameron-Martin space for $B$, and proceed as in \cite{DU97}. Namely define an isomorphism $K$ between $L^{2}([0,1])$
and $I_{0+}^{H+1/2}(L^{2}([0,1]))$ in the following way:
\begin{equation}\label{eq: analytic expression of K}
K\varphi\triangleq\begin{cases}
C_{H}\cdot I_{0^{+}}^{1}\left(t^{H-\frac{1}{2}}\cdot I_{0^{+}}^{H-\frac{1}{2}}\left(s^{\frac{1}{2}-H}\varphi(s)\right)(t)\right), & H>\frac{1}{2};\\
C_{H}\cdot I_{0^{+}}^{2H}\left(t^{\frac{1}{2}-H}\cdot I_{0^{+}}^{\frac{1}{2}-H}\left(s^{H-\frac{1}{2}}\varphi(s)\right)(t)\right), & H\leq\frac{1}{2},
\end{cases}
\end{equation}
where $c_{H}$ is a universal constant depending only on $H.$ One
can easily compute $K^{-1}$ from the definition of $K$ in terms
of fractional derivatives. Moreover, the operator $K$ admits a kernel representation,
i.e. there exits a function $K(t,s)$ such that 
\[
(K\varphi)(t)=\int_{0}^{t}K(t,s)\varphi(s)ds,\ \ \ \varphi\in L^{2}([0,1]).
\]
The kernel $K(t,s)$ is defined for $s<t$ (taking zero value
otherwise). One can write down $K(t,s)$ explicitly thanks to the definitions \eqref{eq:def-frac-integral} and \eqref{eq:def-frac-deriv}, but this expression is not included here since it will not be used later in  our analysis. A crucial property for $K(t,s)$ is that 
\begin{equation}
R(t,s)=\int_{0}^{t\wedge s}K(t,r)K(s,r)dr,\label{eq: kernel representation of covariance function for fBM}
\end{equation} 
where $R(t,s)$ is the fractional Brownian motion covariance function introduced in \eqref{eq:cov-fBm}. This essential fact enables the following analytic characterization of the Cameron-Martin space in \cite[Theorem~3.1]{DU97}.

\begin{thm}\label{thm:bar-cal-H-as frac-integral}
Let $\bch$ be the space given in Definition \ref{def:bar-H}.
As a vector space we have $\bar{{\cal H}}=I_{0^{+}}^{H+1/2}(L^{2}([0,1])),$ and 
the Cameron-Martin norm is given by 
\begin{equation}
\|h\|_{\bar{{\cal H}}}=\|K^{-1}h\|_{L^{2}([0,1])}.\label{eq: inner product in terms of fractional integrals}
\end{equation}
\end{thm}

In order to define Wiener integrals with respect to $B$, it is also convenient to look at the Cameron-Martin subspace in terms of the covariance structure. Specifically, we define another space $\mathcal{H}$ as the completion of the space of simple step functions with inner product induced by 
\begin{equation}\label{eq:def-space-H}
\langle\mathbf{1}_{[0,s]},\mathbf{1}_{[0,t]}\rangle_{{\cal H}}\triangleq R(s,t).
\end{equation}
The space $\ch$ is easily related to $\bch$. Namely define the following operator 
\begin{equation}\label{eq:def-K-star}
{\cal K}^{*}:{\cal H}\rightarrow L^{2}([0,1]),
\quad\text{such that}\quad
\mathbf{1}_{[0,t]}\mapsto K(t,\cdot).
\end{equation}
We also set 
\begin{equation}\label{eq:IsoR}
\mathcal{R}\triangleq K\circ\mathcal{K}^{*}:{\cal H}\rightarrow\bar{\mathcal{H}},
\end{equation}where the operator $K$ is introduced in \eqref{eq: analytic expression of K}. Then it can be proved that $\mathcal{R}$ is an isometric isomorphism (cf. Lemma \ref{lem: surjectivity of K*} below for the surjectivity of $\mathcal{K}^*$). In addition, under this identification, $\mathcal{K}^*$ is the adjoint of $K$, i.e. $\mathcal{K}^{*}=K^{*}\circ\mathcal{R}.$ This can be seen by acting on indicator functions and then taking limits. 
As mentioned above, one advantage about the space $\mathcal{H}$ is that the fractional Wiener integral operator $I:{\mathcal{H}}\rightarrow{\mathcal{C}}_{1}$ induced by  
${\bf 1}_{[0,t]}\mapsto B_{t}$ is an isometric isomorphism. According to relation~\eqref{eq: kernel representation of covariance function for fBM},
$B_{t}$ admits a Wiener integral representation with respect to an underlying Wiener process $W$:
\beq\label{eq:B-as-Wiener-integral}
B_{t}=\int_{0}^{t}K(t,s)dW_{s}  .
\eeq
Moreover, the process $W$ in \eqref{eq:B-as-Wiener-integral} can be expressed as a Wiener integral with respect to $B$, that is
$W_{s}=I(({\cal K}^{*})^{-1}{\bf 1}_{[0,s]})$  (cf. \cite[relation (5.15)]{Nualart06}).

Let us also mention the following useful formula for the natural pairing between $\mathcal{H}$ and $\bar{\mathcal{H}}$. Denote by $C^{H^-}([0,1]; \R^d)$ the space of $\alpha$-H\"{o}lder continuous path for all $\alpha<H$.

\begin{lem}

Let $\ch$ be the space defined as the completion of the indicator functions with respect to the inner product \eqref{eq:def-space-H}. Also recall that $\bch$ is introduced in Definition \ref{def:bar-H}. Then through the isometric isomorphism $\mathcal{R}$ defined by (\ref{eq:IsoR}), the natural pairing between $\ch$ and $\bch$ is given by
\begin{equation}\label{eq: H-barH pairing}
_{\mathcal{H}}\langle f,h\rangle_{\bar{\mathcal{H}}}=\int_{0}^{1}f_{s} \, dh_{s},
\end{equation}
for all $f\in C^{H^-}([0,1];\R^d)$. In the above, the integral on the right-hand side is understood in Young's sense, thanks to Proposition \ref{prop: variational embedding} below.
\end{lem}

\begin{proof}
First of all, let $h\in\bch$ and $g\in\mathcal{H}$ be such that $\mathcal{R}(g)=h$.  It is easy to see that $g$ can be constructed in the following way. According to Definition \ref{def:bar-H}, there exists a random variable $Z$ in the first chaos $\cac_{1}$ such that $h_t=\mathbb{E}[B_tZ]$. 
The element $g\in\mathcal{H}$ is then given via
the Wiener integral isomorphism between $\ch$ and $\cac_{1}$, that is, the element $g\in\ch$ such that $Z=I(g)$. Also note that we have $h_t=\mathbb{E}[B_t \, I(g)]$.

Now consider  $f\in\ch$ with bounded variation. The natural pairing between $f$ and $h$ is thus given by
\begin{equation*}
_{\mathcal{H}}\langle f,h\rangle_{\bar{\mathcal{H}}}
= \,
_{\mathcal{H}}\langle f,g\rangle_{\mathcal{H}}
=
\mathbb{E}[Z\cdot I(f)].
\end{equation*}
A direct application of Fubini's theorem then yields:
\begin{align}\label{pairing for C1 f}
_{\mathcal{H}}\langle f,h\rangle_{\bar{\mathcal{H}}}
=\mathbb{E}[Z\cdot I(f)]
=\mathbb{E}\left[Z\cdot\int_{0}^{1}f_{s}dB_{s}\right]
=\int_{0}^{1}f_{s} \, \mathbb{E}[ZdB_{s}]=\int_{0}^{1}f_{s}dh_{s}.
\end{align}
For a general $f\in\ch$, let $f^n$ be the dyadic linear approximation of $f$. By the previous argument, \eqref{pairing for C1 f} holds for $f^n$. On the other hand, thanks to our H\"older assumption on $f$,  $f^n$ converges to $f$ in $\alpha$-H\"{o}lder norm for any $\alpha<H$. Hence $_{\mathcal{H}}\langle f^n,h\rangle_{\bar{\mathcal{H}}}$ converges to $_{\mathcal{H}}\langle f,h\rangle_{\bar{\mathcal{H}}}$, thanks to \cite[Lemma 5.1.1]{Nualart06} for $H>1/2$ and to the inclusion $C^\gamma([0,1];\R^d)\subset \ch$ for all $\gamma>1/2-H$ (see e.g. \cite[page 284]{Nualart06}) for $H<1/2$. Finally, $\int_0^1f^n_sdh_s$ converges to $\int_0^1f_sdh_s$ by standard Young's estimate. The proof is thus completed.
\end{proof}

The space $\mathcal{H}$ can also be described in terms of fractional calculus (cf. \cite{PT00}), since the operator $\mathcal{K}^*$ defined by \eqref{eq:def-K-star} can be expressed as
\begin{equation}\label{eq: analytic expression for K*}
(\mathcal{K}^{*}f)(t)=\begin{cases}
C_{H}\cdot t^{\frac{1}{2}-H}\cdot\left(I_{1^{-}}^{H-\frac{1}{2}}\left(s^{H-\frac{1}{2}}f(s)\right)\right)(t), & H>\frac{1}{2};\\
C_{H}\cdot t^{\frac{1}{2}-H}\cdot\left(D_{1^{-}}^{\frac{1}{2}-H}\left(s^{H-\frac{1}{2}}f(s)\right)\right)(t), & H\leq\frac{1}{2}.
\end{cases}
\end{equation} 
Starting from this expression, it is readily checked  that when  $H>1/2$ the space $\mathcal{H}$ coincides with the following subspace of the Schwartz distributions $\cs'$:
\begin{equation}\label{eq:ch-case-H-large}
\ch 
=
\lcl f\in\cs' ; \,  t^{1/2-H}\cdot(I_{1^{-}}^{H-1/2}(s^{H-1/2}f(s)))(t) \text{ is an element of } L^{2}([0,1])  \rcl.
\end{equation}
 In the case $H\leq1/2$, we simply have 
 \begin{equation}\label{eq:ch-case-H-small}
\mathcal{H}=I_{1^{-}}^{1/2-H}(L^{2}([0,1])).
\end{equation}

\begin{rem}
As the Hurst parameter $H$ increases, $\mathcal{H}$ gets larger (and contains distributions when $H>1/2$) while $\bar{\mathcal{H}}$ gets smaller. This fact is apparent from Theorem \ref{thm:bar-cal-H-as frac-integral} and relations~\eqref{eq:ch-case-H-large}-\eqref{eq:ch-case-H-small}.
When $H=1/2$, the process $B_t$ coincides with the usual Brownian motion. In this case, we have $\mathcal{H}=L^2([0,1])$ and $\bar{\mathcal{H}}=W_0^{1,2}$, the space of absolutely continuous paths starting at the origin with square integrable derivative.
\end{rem}

Next we mention a variational embedding theorem for the Cameron-Martin subspace $\bar{\mathcal{H}}$ which will be used in a crucial way. The case when $H>1/2$ is a simple exercise starting from the definition \eqref{eq:def-h-in-CM} of $\bar{\mathcal{H}}$ and invoking the Cauchy-Schwarz inequality. The case when $H\leq 1/2$ was treated in \cite{FV06}. From a pathwise point of view, this allows us to integrate a fractional Brownian path against a Cameron-Martin path or vice versa (cf. \cite{Young36}), and to make sense of ordinary differential equations driven by a Cameron-Martin path (cf. \cite{Lyons94}).

\begin{prop}\label{prop: variational embedding}If $H>\frac{1}{2}$, then
$\bar{\mathcal{H}}\subseteq C_{0}^{H}([0,1];\mathbb{R}^{d})$, the space of H-H\"older continuous paths. If $H\leq\frac{1}{2}$,
then for any $q>\left(H+1/2\right)^{-1}$, we have 
$\bar{\mathcal{H}}\subseteq C_{0}^{q\text{-}\mathrm{var}}([0,1];\mathbb{R}^{d})$, the space of continuous paths with finite $q$-variation. In addition, the above inclusions are continuous embeddings.
\end{prop}

Finally, we prove two general lemmas on the Cameron-Martin subspace that are needed later on. These properties do not seem to be contained in the literature and they require some care based on fractional calculus.  The first one claims the surjectivity of $\ck^{*}$ on properly defined spaces.

\begin{lem}\label{lem: surjectivity of K*}
Let $H\in(0,1)$, and consider 
the operator $\mathcal{K}^*:\mathcal{H}\rightarrow L^2([0,1])$ defined by~\eqref{eq:def-K-star}. Then $\ck^{*}$ is surjective.
\end{lem}

\begin{proof}
If $H>1/2$, we know that the image of $\mathcal{K}^*$ contains all indicator functions (cf. \cite[Equation (5.14)]{Nualart06}). Therefore, $\mathcal{K}^*$ is surjective. 

If $H<1/2$, we first claim that the image of $\mathcal{K}^*$ contains functions of the form $t^{1/2-H}p(1-t)$ where $p(t)$ is a polynomial. 
Indeed, given an arbitrary $\beta\geq0,$ consider the function
$$f_{\beta}(t)\triangleq t^{\frac{1}{2}-H}(1-t)^{\beta+\frac{1}{2}-H}.$$ It is readily checked that $D_{1^{-}}^{\frac{1}{2}-H}f_{\beta}\in L^{2}([0,1])$, and hence $f_{\beta}\in I_{1^{-}}^{\frac{1}{2}-H}(L^{2}([0,1]))=\mathcal{H}$. Using the analytic expression (\ref{eq: analytic expression for K*}) for $\mathcal{K}^*$, we can compute $\mathcal{K}^*f_\beta$ explicitly (cf. \cite[Chapter 2, Equation (2.45)]{KMS93}) as\[
(\mathcal{K}^{*}f_{\beta})(t)=C_{H}\frac{\Gamma\left(\beta+\frac{3}{2}-H\right)}{\Gamma(\beta+1)}t^{\frac{1}{2}-H}(1-t)^{\beta}.
\] Since $\beta$ is arbitrary and $\mathcal{K}^*$ is linear, the claim follows.

Now it remains to show (with a change of variable) that the space of  functions of the form $(1-t)^{\frac{1}{2}-H}p(t)$ with $p(t)$ being a polynomial is dense in $L^2([0,1])$. To this end, let $\varphi\in C_c^\infty((0,1))$. Then $\psi(t)\triangleq (1-t)^{-(1/2-H)}\varphi(t)\in C_{c}^{\infty}((0,1)).$ According to Bernstein's approximation theorem, for any $\varepsilon>0,$ there exists a polynomial $p(t)$ such that \[
\|\psi-p\|_{\infty}<\varepsilon,
\]and thus\[
\sup_{0\leq t\leq1}|\varphi(t)-(1-t)^{\frac{1}{2}-H}p(t)|<\varepsilon.
\]Therefore, functions in $C_c^\infty((0,1))$ (and thus in $L^2([0,1])$) can be approximated by functions of the desired form.
\end{proof}

Our second lemma gives some continuous embedding properties for $\ch$ and $\bch$ in the irregular case $H<1/2$, whose proof relies on Lemma \ref{lem: surjectivity of K*}.

\begin{lem}\label{lem: continuous embedding when H<1/2}
For $H<1/2$, the inclusions $\mathcal{H}\subseteq L^2([0,1])$ and $W_0^{1,2}\subseteq \bar{\mathcal{H}}$ are continuous embeddings.
\end{lem}

\begin{proof}

For the first assertion, let $f\in\mathcal{H}$. We wish to prove that
\begin{equation}\label{eq: continuity of H to L^2}
\|f\|_{L^{2}([0,1])}\leq C_{H}\|f\|_{\mathcal{H}}.
\end{equation}
Towards this aim, define $\varphi\triangleq \mathcal{K}^*f$, where $\ck^{*}$ is defined by \eqref{eq:def-K-star}. Observe that $\ck^{*}:\ch\to L^2([0,1])$ and thus $f\in L^2([0,1])$. By solving $f$ in terms of $\varphi$ using the analytic expression~\eqref{eq: analytic expression for K*} for $\mathcal{K}^*$, we have 
\beq\label{a1}
f(t)=C_H t^{\frac{1}{2}-H}\left(I_{1-}^{\frac{1}{2}-H}\left(s^{H-\frac{1}{2}}\varphi(s)\right)\right)(t).
\eeq
We now bound the right hand side of \eqref{a1}. Our first step in this direction is to notice that according to the definition \eqref{eq:def-frac-integral} of fractional integral we have
\begin{align*}
 \left|\left(I_{1^{-}}^{\frac{1}{2}-H}(s^{H-\frac{1}{2}}\varphi(s))\right)(t)\right|
 & =C_{H}\left|\int_{t}^{1}(s-t)^{-\frac{1}{2}-H}s^{H-\frac{1}{2}}\varphi(s)ds\right|\\
 & \leq C_{H}\int_{t}^{1}(s-t)^{-\frac{1}{2}-H}s^{H-\frac{1}{2}}|\varphi(s)|ds\\
 & =C_{H}\int_{t}^{1}(s-t)^{-\frac{1}{4}-\frac{H}{2}} \left((s-t)^{-\frac{1}{4}-\frac{H}{2}}s^{H-\frac{1}{2}}|\varphi(s)|\right)ds .
 \end{align*}
Hence a direct application of Cauchy-Schwarz inequality gives
\begin{align}\label{a2}
  \left|\left(I_{1^{-}}^{\frac{1}{2}-H}(s^{H-\frac{1}{2}}\varphi(s))\right)(t)\right|
& \leq C_{H}\left(\int_{t}^{1}(s-t)^{-\frac{1}{2}-H}ds\right)^{\frac{1}{2}}
 \left(\int_{t}^{1}(s-t)^{-\frac{1}{2}-H}s^{2H-1}|\varphi(s)|^{2}ds\right)^{\frac{1}{2}} \notag\\
 & =C_{H} (1-t)^{\frac{1}{2}\lp\frac{1}{2}-H\rp} 
 \left(\int_{t}^{1}(s-t)^{-\frac{1}{2}-H}s^{2H-1}|\varphi(s)|^{2}ds\right)^{\frac{1}{2}},
\end{align}
where we recall that $C_{H}$ is a positive constant which can change from line to line.
Therefore, plugging \eqref{a2} into \eqref{a1} we obtain
\begin{equation*}
\|f\|_{L^{2}([0,1])}^{2}  
\leq 
C_{H}\int_{0}^{1} \lp t^{1-2H}(1-t)^{\frac{1}{2}-H}\int_{t}^{1}(s-t)^{-\frac{1}{2}-H}s^{2H-1}|\varphi(s)|^{2}ds \rp dt.
\end{equation*}
We now bound all the terms of the form $s^{\beta}$ with $\beta>0$ by 1. This gives
\begin{align*}
\|f\|_{L^{2}([0,1])}^{2} 
 & \leq C_{H}\int_{0}^{1}dt\int_{t}^{1}(s-t)^{-\frac{1}{2}-H}|\varphi(s)|^{2}ds
  =C_{H}\int_{0}^{1}|\varphi(s)|^{2}ds\int_{0}^{s}(s-t)^{-\frac{1}{2}-H}dt \\
 & =C_{H}\int_{0}^{1}s^{\frac{1}{2}-H}  |\varphi(s)|^{2}ds
  \leq C_{H}\|\varphi\|_{L^{2}([0,1])}^{2}
  =C_{H}\|f\|_{\mathcal{H}}^{2},
\end{align*}
which is our claim \eqref{eq: continuity of H to L^2}.

For the second assertion about the embedding of $W_0^{1,2}$ in $\bar{\mathcal{H}}$, let $h\in W_0^{1,2}$. We thus also have  $h\in\bch$ and we can write $h=K\varphi$ for some $\varphi\in L^2([0,1])$. We first claim that 
\begin{equation}\label{eq: duality}
\int_{0}^{1}f(s)dh(s)=\int_{0}^{1}\mathcal{K}^{*}f(s)\varphi(s)ds
\end{equation}
for all $f\in\mathcal{H}$. This assertion can be reduced in the following way: since $\ch\hookrightarrow L^{2}([0,1])$ continuously and $\ck^{*}:\ch\to L^{2}([0,1])$ is continuous, one can take limits along indicator functions in \eqref{eq: duality}. Thus it is sufficient to consider $f=\mathbf{1}_{[0,t]}$ in \eqref{eq: duality}. In addition, relation~\eqref{eq: duality} can be checked easily for $f=\mathbf{1}_{[0,t]}$. Namely we have
\[
\int_{0}^{1}\mathbf{1}_{[0,t]}(s)dh(s)=h(t)=\int_{0}^{t}K(t,s)\varphi(s)ds=\int_{0}^{1}\left(\mathcal{K}^{*}\mathbf{1}_{[0,t]}\right)(s)\varphi(s)ds.
\]
Therefore, our claim  (\ref{eq: duality}) holds true. Now from Lemma \ref{lem: surjectivity of K*}, if $\vp\in L^{2}([0,1])$ there exists $f\in\mathcal{H}$ such that $\varphi=\mathcal{K}^*f$. For this particular $f$, invoking relation~\eqref{eq: duality} we get
\beq\label{a21}
\int_{0}^{1}f(s)dh(s)=\|\varphi\|_{L^{2}([0,1])}^{2}.
\eeq
But we also know that
\beq\label{a3}
\|\varphi\|_{L^{2}([0,1])}=\|h\|_{\bar{\mathcal{H}}}=\|f\|_{\mathcal{H}},
\quad\text{and thus}\quad
\|\varphi\|_{L^{2}([0,1])}^{2} = \|h\|_{\bar{\mathcal{H}}}  \|f\|_{\mathcal{H}}.
\eeq
In addition recall that the $W^{1,2}$ norm can be written as
\begin{equation*}
\|h\|_{W^{1,2}}  =\sup_{\psi\in L^{2}([0,1])}\frac{\left|\int_{0}^{1}\psi(s)dh(s)\right|}{\|\psi\|_{L^{2}([0,1])}}
\end{equation*}
Owing to \eqref{a21} and \eqref{a3} we thus get
\begin{align*}
\|h\|_{W^{1,2}} & 
\geq\frac{\int_{0}^{1}f(s)dh(s)}{\|f\|_{L^{2}([0,1])}}
 =\frac{\|h\|_{\bar{\mathcal{H}}}\|f\|_{\mathcal{H}}}{\|f\|_{L^{2}([0,1])}}\geq C_{H}\|h\|_{\bar{\mathcal{H}}},
\end{align*}
where the last step stems from (\ref{eq: continuity of H to L^2}). The continuous embedding $W_0^{1,2}\subseteq \bar{\mathcal{H}}$ follows.
\end{proof}

\subsection{Malliavin calculus for fractional Brownian motion.}

In this section we review some basic aspects of Malliavin calculus and set up  corresponding notations. The reader is referred to~\cite{Nualart06} for further details.

We consider the fractional Brownian motion $B=(B^1,\ldots,B^d)$ as in Definition \eqref{def:fbm}, defined on a complete probability space $(\Omega, \cf, \mathbb{P})$. For sake of simplicity, we assume that $\cf$ is generated by $\{B_{t}; \, t\in[0,T]\}$. An $\mathcal{F}$-measurable real
valued random variable $F$ is said to be \textit{cylindrical} if it can be
written, with some $m\ge 1$, as
\begin{equation*}
F=f\lp  B_{t_1},\ldots,B_{t_m}\rp,
\quad\mbox{for}\quad
0\le t_1<\cdots<t_m \le 1,
\end{equation*}
where $f:\mathbb{R}^m \rightarrow \mathbb{R}$ is a $C_b^{\infty}$ function. The set of cylindrical random variables is denoted by~$\mathcal{S}$. 

\smallskip

The Malliavin derivative is defined as follows: for $F \in \mathcal{S}$, the derivative of $F$ in the direction $h\in\ch$ is given by
\[
\mathbf{D}_h F=\sum_{i=1}^{m}  \frac{\partial f}{\partial
x_i} \left( B_{t_1},\ldots,B_{t_m}  \right) \, h_{t_i}.
\]
More generally, we can introduce iterated derivatives. Namely, if $F \in
\mathcal{S}$, we set
\[
\mathbf{D}^k_{h_1,\ldots,h_k} F = \mathbf{D}_{h_1} \ldots\mathbf{D}_{h_k} F.
\]
For any $p \geq 1$, it can be checked that the operator $\mathbf{D}^k$ is closable from
$\mathcal{S}$ into $\mathbf{L}^p(\oom;\ch^{\otimes k})$. We denote by
$\mathbb{D}^{k,p}(\ch)$ the closure of the class of
cylindrical random variables with respect to the norm
\begin{equation} \label{eq:malliavin-sobolev-norm}
\left\| F\right\| _{k,p}=\left( \mathbb{E}\left[|F|^{p}\right]
+\sum_{j=1}^k \mathbb{E}\left[ \left\| \mathbf{D}^j F\right\|
_{\ch^{\otimes j}}^{p}\right] \right) ^{\frac{1}{p}},
\end{equation}
and we also set $\mathbb{D}^{\infty}(\ch)=\cap_{p \geq 1} \cap_{k\geq 1} \mathbb{D}^{k,p}(\ch)$. 

\smallskip

Estimates of Malliavin derivatives are crucial in order to get information about densities of random variables, and Malliavin covariance matrices as well as non-degenerate random variables will feature importantly in the sequel.
\begin{defn}\label{non-deg}
Let $F=(F^1,\ldots , F^n)$ be a random vector whose components are in $\mathbb{D}^\infty(\ch)$. Define the \textit{Malliavin covariance matrix} of $F$ by
\begin{equation} \label{malmat}
\gamma_F=(\langle \mathbf{D}F^i, \mathbf{D}F^j\rangle_{\ch})_{1\leq i,j\leq n}.
\end{equation}
Then $F$ is called  {\it non-degenerate} if $\gamma_F$ is invertible $a.s.$ and
$$(\det \gamma_F)^{-1}\in \cap_{p\geq1}L^p(\Omega).$$
\end{defn}
It is a classical result that the law of a non-degenerate random vector $F=(F^1, \ldots , F^n)$ admits a smooth density with respect to the Lebesgue measure on $\mr^n$. 

\section{Proof of main results.}\label{sec: ellip.}

In this section, we prove Theorem~\ref{thm: local comparison} and Theorem~\ref{thm: local lower estimate}. We emphasize that our analysis relies crucially on the uniform ellipticity of the vector fields $V_i$'s in equation \eqref{sde-intro}, which is spelled out explicitly below.
\begin{uell*}
The $\cac_{b}^{\infty}$ vector fields $V=\{V_1,\ldots,V_d\}$ are such that
\beq\label{eq:hyp-elliptic}
\Lambda_{1}|\xi|^{2}\leq\xi^{*}V(x)V(x)^{*}\xi\leq\Lambda_{2}|\xi|^{2},\ \ \ \forall x,\xi\in\mathbb{R}^{N},
\eeq
with some constants $\Lambda_1,\Lambda_2>0$, where $(\cdot)^*$ denotes matrix transpose. 
\end{uell*}

\noindent
We now split our proofs in two subsections, corresponding respectively to Theorem~\ref{thm: local comparison} and Theorem~\ref{thm: local lower estimate}.

\subsection{Proof of the distance comparison}
In order to prove Theorem~\ref{thm: local comparison},
recall first that $\Phi_t(x;\cdot): \bar{\mathcal{H}}\to C[0,1]$ is the deterministic It\^{o} map associated to equation~\eqref{sde-intro}. For $x,y\in\mathbb{R}^N$, set
\begin{equation}\label{eq: Pi_xy}
\Pi_{x,y}\triangleq\left\{ h\in\bar{\mathcal{H}}:\Phi_{1}(x;h)=y\right\} 
\end{equation}
the  set of Cameron-Martin paths that joining $x$ to $y$ though the It\^{o} map. Under our assumption \eqref{eq:hyp-elliptic} it is easy to construct an $h\in\bar{\mathcal{H}}\in\Pi_{x,y}$ explicitly, which will ease our computations later on.

\begin{lem}\label{lem: explicit construction of joining h}
Let $V=\{V_1,\ldots,V_d\}$ be vector fields satisfying the uniform elliptic assumption~\eqref{eq:hyp-elliptic}. Given $x,y\in\mathbb{R}^N$, define 
\beq\label{eq:def-ht-elliptic-case}
h_{t}\triangleq\int_{0}^{t}V^{*}(z_{s})\cdot(V(z_{s})V^{*}(z_{s}))^{-1}\cdot(y-x)ds ,
\eeq
where $z_{t}\triangleq(1-t)x+ty$ is the line segment from $x$ to $y$. Then $h\in\Pi_{x,y}$, where $\Pi_{x,y}$ is defined by relation \eqref{eq: Pi_xy}. 
\end{lem}

\begin{proof}
Since $\bar{\mathcal{H}}=I_{0^+}^{H+1/2}(L^2([0,1]))$ contains smooth paths, it is obvious that $h\in\bar{\mathcal{H}}$. As far as $z_{t}$ is concerned, the definition $z_{t}=(1-t)x+ty$ clearly implies that $z_0=x,z_1=y$ and $\dot{z}_{t}=y-x$. In addition, since $VV^{*}(\xi)$ is invertible for all $\xi\in\R^{N}$ under our condition \eqref{eq:hyp-elliptic}, we get
\[
\dot{z}_{t}=y-x=\left(VV^{*}(VV^{*})^{-1}\right)(z_{t})\cdot(y-x)=V(z_{t})\dot{h}_{t},
\] 
where the last identity stems from the definition \eqref{eq:def-ht-elliptic-case} of $h$.
Therefore $h\in\Pi_{x,y}$ according to our definition \eqref{eq: Pi_xy}.
\end{proof}

\begin{rem}
The intuition behind Lemma \ref{lem: explicit construction of joining h} is very simple. Indeed, given any smooth path $x_t$ with $x_0=x,x_1=y$, since the vector fields are elliptic, there exist smooth functions $\lambda^1(t),\ldots,\lambda^d (t)$, such that\[
\dot{x}_{t}=\sum_{\alpha=1}^{d}\lambda^{\alpha}(t)V_{\alpha}(x_{t}),\ \ \ 0\leq t\leq1.
\]In matrix notation, $\dot{x}_t=V(x_t)\cdot\lambda(t)$. A canonical way to construct $\lambda(t)$ is writing it as $\lambda(t)=V^*(x_t)\eta(t)$ so that from ellipticity we can solve for $\eta(t)$ as $$\eta(t)=(V(x_{t})V^{*}(x_{t}))^{-1}\dot{x}_{t}.$$
It follows that the path $h_t\triangleq\int_0^t \lambda(s)ds$ belongs to $\Pi_{x,y}$.
\end{rem}

Now we can prove the following result which asserts that the control distance function is locally comparable with the Euclidean metric, that is  Theorem \ref{thm: local lower estimate} under elliptic assumptions.

\begin{thm}\label{thm: local comparison in elliptic case}
Let $V=\{V_1,\ldots,V_d\}$ be vector fields satisfying the uniform elliptic assumption \eqref{eq:hyp-elliptic}. Consider  the control distance $d$ given in \eqref{def: distance} for a given $H>\frac14$.
Then there exist constants $C_1,C_2>0$ depending only on $H$ and the vector fields, such that 
\beq\label{eq:local-euclid-d-elliptic}
C_{1}|x-y|\leq d(x,y)\leq C_{2}|x-y|
\eeq
for all $x,y\in\mathbb{R}^N$ with $|x-y|\leq1.$
\end{thm}

\begin{proof}

We first consider the case when $H\leq1/2$, which is simpler due to Lemma \ref{lem: continuous embedding when H<1/2}. Given $x,y\in\mathbb{R}^N$, define $h\in\Pi_{x,y}$ as in Lemma \ref{lem: explicit construction of joining h}. According to Lemma \ref{lem: continuous embedding when H<1/2} 
and \eqref{def: distance} we have
\[
d(x,y)^{2}\leq\|h\|_{\bar{\mathcal{H}}}^{2}\leq C_{H}\|h\|_{W^{1,2}}^{2}.\]
Therefore, according to the definition (\ref{eq:def-ht-elliptic-case}) of $h$, we get
\[
d(x,y)^{2}\leq C_{H}\int_{0}^{1}|V^*(z_s)(V(z_{s})V^*(z_s))^{-1}\cdot(y-x)|^{2}ds\leq C_{H,V}|y-x|^{2},
\]
where the last inequality stems from the uniform ellipticity assumption \eqref{eq:hyp-elliptic} and the fact that $V^*$ is bounded. This proves the upper bound in \eqref{eq:local-euclid-d-elliptic}.

We now turn to the lower bound in \eqref{eq:local-euclid-d-elliptic}. To this aim, consider $h\in\Pi_{x,y}$. We assume (without loss of generality) in the sequel that
\begin{equation}\label{b1}
\|h\|_{\bar{\mathcal{H}}}
\leq
2 d(x,y)
\leq
2 C_{2},
\end{equation}
where the last inequality is due to the second part of inequality \eqref{eq:local-euclid-d-elliptic} and the fact that $|x-y|\leq1$.
Then recalling the definition \eqref{eq: Pi_xy} of $\Pi_{x,y}$ we have 
\[
y-x=\int_{0}^{1}V(\Phi_{t}(x;h))dh_{t}.
\]
According to Proposition \ref{prop: variational embedding} (specifically the embedding $\bar{\mathcal{H}}\subseteq C_{0}^{q-\mathrm{var}}([0,1];\mathbb{R}^{d})$ for $q>(H+1/2)^{-1}$) and the pathwise variational estimate given by \cite[Theorem 10.14]{FV10},   we have 
\begin{equation}\label{b2}
|y-x|  \leq C_{H,V} \lp \|h\|_{q-{\rm var}}\vee\|h\|_{q-{\rm var}}^{q} \rp
  \leq C_{H,V} \lp \|h\|_{\bar{\mathcal{H}}}\vee\|h\|_{\bar{\mathcal{H}}}^{q} \rp.
\end{equation}
Since $q\geq1$ and owing to \eqref{b1}, we conclude that 
\[
|y-x|\leq C_{H,V}\|h\|_{\bar{\mathcal{H}}}
\] for all $x,y$ with $|y-x|\leq1$. Since $h\in\Pi_{x,y}$ is arbitrary provided \eqref{b1} holds true,  the lower bound in \eqref{eq:local-euclid-d-elliptic} follows again by a direct application of \eqref{def: distance}.

Next we consider the case when $H>1/2$. The lower bound in \eqref{eq:local-euclid-d-elliptic} can be proved with the same argument as in the case $H\leq1/2$, 
the only difference being that in \eqref{b2} we replace $\bar{\mathcal{H}}\subseteq C_{0}^{q-\mathrm{var}}([0,1];\mathbb{R}^{d})$ by $\bar{\mathcal{H}}\subseteq C_{0}^{H}([0,1];\mathbb{R}^{d})$ and the pathwise variational estimate of \cite[Theorem 10.14]{FV10} by a
H\"older estimate borrowed from  \cite[Proposition 8.1]{FH14}. 

For the upper bound in \eqref{eq:local-euclid-d-elliptic}, we again take $h\in\Pi_{x,y}$ as given by Lemma \ref{lem: explicit construction of joining h} and estimate its Cameron-Martin norm. Note that due to our uniform ellipticity assumption \eqref{eq:hyp-elliptic}, one can define the function 
\begin{align}\label{eq: def gamma}
\gamma_t\equiv\int_0^t (V^*(VV^*)^{-1})(z_s)ds=\int_0^t g((1-s)x+sy)ds,\end{align} 
where $g$ is a matrix-valued $C_b^\infty$ function. We will now prove that $\gamma$ can be written as $\gamma=K\varphi$ for $\varphi\in L^2([0,1])$. Indeed, one can solve for $\varphi$ in the analytic expression (\ref{eq: analytic expression of K}) for $H>1/2$ and get 
\begin{align*}
\varphi(t) & =C_{H}t^{H-\frac{1}{2}}\left(D_{0^{+}}^{H-\frac{1}{2}}\left(s^{\frac{1}{2}-H}\dot{\gamma}_{s}\right)\right)(t).
\end{align*}
We now use the expression \eqref{eq:def-frac-deriv} for $D^{H-1/2}_{0^+}$, which yield (after an elementary change of variable)
\begin{align*}
 \varphi(t)& =C_{H}t^{H-\frac{1}{2}}\frac{d}{dt}\int_{0}^{t}s^{\frac{1}{2}-H}(t-s)^{\frac{1}{2}-H}g((1-s)x+sy)ds\\
 & =C_{H}t^{H-\frac{1}{2}}\frac{d}{dt}\left(t^{2-2H}\int_{0}^{1}(u(1-u))^{\frac{1}{2}-H}g((1-tu)x+tuy)du\right)\\
 & =C_{H}t^{\frac{1}{2}-H}\int_{0}^{1}(u(1-u))^{\frac{1}{2}-H}g((1-tu)x+tuy)du\\
 & \ \ \ +C_{H}t^{\frac{3}{2}-H}\int_{0}^{1}(u(1-u))^{\frac{1}{2}-H}u\nabla g((1-tu)x+tuy)\cdot(y-x)du.
\end{align*}
Hence, thanks to the fact that $g$ and $\nabla g$ are bounded plus the fact that $t\leq1$, we get 
\[
|\varphi(t)|\leq C_{H,V}(t^{\frac{1}{2}-H}+|y-x|),
\]from which $\varphi$ is clearly an element of $L^2([0,1])$. Since $|y-x|\leq 1$, we conclude that \[
\|\gamma\|_{\bar{\mathcal{H}}}=\|\varphi\|_{L^{2}([0,1])}\leq C_{H,V}.
\] Therefore, recalling that $h$ is given by \eqref{eq:def-ht-elliptic-case} and $\gamma$ is defined by \eqref{eq: def gamma}, we end up with
\begin{align*}
d(x,y)\leq\|h\|_{\bar{\mathcal{H}}}&=\left\|\left(\int_{0}^{\cdot}(V^*(VV^*)^{-1})(z_{s})ds\right)\cdot(y-x)\right\|_{\bar{\mathcal{H}}}\\
&=\|\gamma\|_{\bar{\mathcal{H}}}|y-x|\leq C_{H,V}|y-x|.
\end{align*}
This concludes the proof.
\end{proof}


\subsection{Lower bounds for the density}

With Theorem \ref{thm: local comparison in elliptic case} in hand, we are now ready to state Theorem~\ref{thm: local lower estimate} rigorously and prove it. Specifically, our main local bound on the density of $X_{t}$ takes the following form.

\begin{thm}\label{thm: local lower estimate in elliptic case}
Let $p(t,x,y)$ be the density of the solution $X_t$ to equation \eqref{sde-intro}. Under the uniform ellipticity assumption (\ref{eq:hyp-elliptic}), there exist constants $C_1,C_2,\tau>0$ depending only on $H$ and the vector fields $V$, such that \begin{align}\label{local lower estimate-elliptic}p(t,x,y)\geq\frac{C_1}{t^{NH}}\end{align} for all $(t,x,y)\in(0,1]\times\mathbb{R}^N\times\mathbb{R}^N$ satisfying $|x-y|\leq C_2 t^H$ and $t<\tau$.
\end{thm}

\begin{rem}
From Theorem \ref{thm: local comparison in elliptic case}, we know that $|B_d(x,t^H)|\asymp t^{NH}$ when $t$ is small. Therefore, Theorem \ref{thm: local lower estimate} becomes the following result, which is consistent with the intuition that the density $p(t,x,y)$ of the solution to equation \eqref{sde-intro} should behave like the Gaussian kernel:
\[
p(t,x,y)\asymp\frac{C_{1}}{t^{NH}}\exp\left(-\frac{C_{2}|y-x|^{2}}{t^{2H}}\right).
\]
\end{rem}

The main idea behind the proof of Theorem \ref{thm: local lower estimate in elliptic case} is to translate the small time  estimate in \eqref{local lower estimate-elliptic} into a large deviation estimate. To this aim, we will first recall some preliminary notions taken from \cite{BOZ15}. By a slight abuse of notation, for any sample path $w$ of $B$ we will call  $w\mapsto\Phi_t(x;w)$ the solution map of the SDE (\ref{sde-intro}). From the scaling invariance of fractional Brownian motion, it is not hard to see that \begin{align}\label{scaling property of equation driven by fBm}
\Phi_{t}(x;B)\stackrel{{\rm law}}{=}\Phi_{1}(x;\varepsilon B),
\end{align} where $\varepsilon\triangleq t^H$. Therefore, since the random variable $\Phi_{t}(x;B)$ is nondegenerate under our standing assumption~\eqref{eq:hyp-elliptic}, the density $p(t,x,y)$ can be written as
\begin{equation}\label{b3}
p(t,x,y)=\mathbb{E}\left[\delta_{y}\left(\Phi_{1}(x;\varepsilon B)\right)\right].
\end{equation}

Starting from expression \eqref{b3}, we now label a proposition which gives a lower bound on $p(t,x,y)$ in terms of some conveniently chosen shifts on the Wiener space.
\begin{prop}\label{th: summary of 4}
In this proposition, $\Phi_t$ stands for  the solution map of equation \eqref{sde-intro}.  The vector fields $\{V_1,\ldots,V_d\}$ are supposed to satisfy the uniform elliptic assumption~\eqref{eq:hyp-elliptic}. Then the following holds true.

\vspace{2mm}
\noindent (i) Let $\Phi_t$ be the solution map of equation \eqref{sde-intro}, $h\in\bar{\mathcal{H}}$, and let
\begin{align}\label{def approx X(h)}
X^{\varepsilon}(h)\triangleq\frac{\Phi_{1}(x;\varepsilon B+h)-\Phi_{1}(x;h)}{\varepsilon}.
\end{align}
Then  $X^\ep(h)$ converges in $\mathbb{D}^\infty$ to X(h) uniformly in $h\in\bar{\mathcal{H}}$ with $\|h\|_{\bar{\mathcal{H}}}\leq M$ (for any $M>0$). 
Moreover $X(h)$ is a $\R^N$-valued centered Gaussian random variable whose covariance matrix will be specified below.\\
(ii) Let $\ep>0$ and consider $x,y\in\R^N$ such that $d(x,y)\leq\ep$, where $d(\cdot,\cdot)$ is the distance considered in Theorem \ref{thm: local comparison in elliptic case}.  Choose $h\in\Pi_{x,y}$ so that 
\begin{equation}\label{eq:local-cdt-on-h}
\|h\|_{\bar{\mathcal{H}}}\leq d(x,y)+\varepsilon\leq2\varepsilon.
\end{equation}
Then we have
\begin{align}\label{eq: CM theorem}
\mathbb{E}\left[\delta_{y}\left(\Phi_{1}(x;\varepsilon B)\right)\right]  \geq C\varepsilon^{-N}\cdot\mathbb{E}\left[\delta_{0}\left(X^\ep(h)\right){\rm e}^{-I\left(\frac{h}{\varepsilon}\right)}\right].
\end{align}
\end{prop}

\begin{proof}
The first statement is proved in \cite{BOZ15}. For the second statement, according to the Cameron-Martin theorem, we have 
\begin{equation*}
\mathbb{E}\left[\delta_{y}\left(\Phi_{1}(x;\varepsilon B)\right)\right]  
=
{\rm e}^{-\frac{\|h\|_{\bar{\mathcal{H}}}^{2}}{2\varepsilon^{2}}}
\mathbb{E}\left[\delta_{y}\left(\Phi_{1}(x;\varepsilon B+h)\right){\rm e}^{-I\left(\frac{h}{\varepsilon}\right)}\right],
\end{equation*}
where we have identified $\bar{\mathcal{H}}$ with $\mathcal{H}$ through $\mathcal{R}$ and recall that $I:\mathcal{H}\rightarrow\mathcal{C}_1$ is the Wiener integral operator introduced in Section \ref{sec: prel.}.
Therefore, thanks to inequality \eqref{eq:local-cdt-on-h}, we get
\begin{equation*}
\mathbb{E}\left[\delta_{y}\left(\Phi_{1}(x;\varepsilon B)\right)\right]  
\geq C\cdot\mathbb{E}\left[\delta_{y}\left(\Phi_{1}(x;\varepsilon B+h)\right){\rm e}^{-I\left(\frac{h}{\varepsilon}\right)}\right].
\end{equation*}
 In addition we have chosen $h\in\Pi_{x,y}$, which means that $\Phi_1(x;h)=y$. Thanks to the scaling property of the Dirac delta function in $\R^N$, we get
 \begin{align*}
p(t,x,y)=\mathbb{E}\left[\delta_{y}\left(\Phi_{1}(x;\varepsilon B)\right)\right]  
&\ge 
C\varepsilon^{-N}\cdot\mathbb{E}\left[\delta_{0}\left(\frac{\Phi_{1}(x;\varepsilon B+h)-\Phi_{1}(x;h)}{\varepsilon}\right){\rm e}^{-I\left(\frac{h}{\varepsilon}\right)}\right].
\end{align*}
Our claim  \eqref{eq: CM theorem} thus follows from the definition \eqref{def approx X(h)} of $X^{\ep}(h)$.
\end{proof}

Let us now describe the covariance matrix of $X(h)$ introduced in Proposition \ref{th: summary of 4}. For this, we recall again that $\Phi$ is the application defined on $\bar{\mathcal{H}}$ by the deterministic It\^{o} map associated to \eqref{sde-intro}. The Jacobian of $\Phi_t(\cdot\,;h): \R^N\to\R^N$ is denoted by $J(\cdot\,;h)$. 

First, it is easy to see that the deterministic Malliavin differential $D^l\Phi_t:=\langle D\Phi_t(x,h), l\rangle_{\bar{\mathcal{H}}}$ of $\Phi$ satisfies
\begin{align}\label{eq for Malliavin diff}
D^l\Phi_t=\sum_{i=1}^d\int_0^t\partial V_i(\Phi_s(x;h))D^l\Phi_sdh^i_s+\sum_{i=1}^d\int_0^tV_i(\Phi_s(x;h))dl^i_s, \ \ \ \mathrm{for\ all}\  l\in\bar{\mathcal{H}},
\end{align}
where $D$ is the Malliavin derivative operator.  Comparing \eqref{eq for Malliavin diff} to the equation satisfied by $J(x;h)$, it is standard from ODE theory that 
\begin{equation}\label{eq:derivative-Phi-with-Jacobian}
\langle D\Phi_{t}(x;h),l\rangle_{\bar{\mathcal{H}}}=J_{t}(x;h)\cdot\int_{0}^{t}J_{s}^{-1}(x;h)\cdot V(\Phi_{s}(x;h))dl_{s}.
\end{equation} 
According to the pairing (\ref{eq: H-barH pairing}), when viewed as an $\mathcal{H}$-valued functional, we have 
\begin{align}\label{eq: expression Malliavin derivative}(D\Phi_{t}^{i}(x;h))_{s}=\left(J_{t}(x;h)J_{s}^{-1}(x;h)V(\Phi_{s}(x;h))\right)^{i}\mathbf{1}_{[0,t]}(s),\ \ \ 1\leq i\leq N.
\end{align}
Next, observe that the Malliavin differential of $X_t(h):=\lim_{\varepsilon\downarrow 0}(\Phi(x;\varepsilon B+h)-\Phi_t(x;h))/\varepsilon$ satisfies the same equation as \eqref{eq for Malliavin diff}, which is deterministic. This implies that $X_t(h)$ is a Gaussian random variable and  the $N\times N$ covariance matrix of $X_t(h)$ admits the following representation 
\begin{align}\label{Malliavin matrix X(h)}
\mathrm{Cov}(X_t(h))\equiv \Gamma_{\Phi_t(x;h)}=\langle D\Phi_t(x;h), D\Phi_t(x;h)\rangle_\mathcal{H}.
\end{align}


With \eqref{Malliavin matrix X(h)} in hand, a crucial point for proving Theorem \ref{thm: local lower estimate in elliptic case} is the fact that $\Gamma_{\Phi_1(x;h)}$ is uniformly non-degenerate with respect to all $h$. This is the content of the following result which is another special feature of ellipticity that fails in the hypoelliptic case. Its proof is an adaptation of the argument in \cite{BOZ15} to the deterministic context.

\begin{lem}\label{lem: uniform nondegeneracy of Malliavin matrix}
Let $M>0$ be a localizing constant. Consider the Malliavin covariance matrix $\Gamma_{\Phi_1(x;h)}$ defined by \eqref{Malliavin matrix X(h)}. Under the uniform ellipticity assumption (\ref{eq:hyp-elliptic}), there exist $C_1,C_2>0$ depending only on $H,M$ and the vector fields, such that 
\begin{align}\label{uniform bound for Malliavin matrix}C_{1}\leq\det \Gamma_{\Phi_{1}(x;h)}\leq C_{2}
\end{align}for all $x\in\mathbb{R}^N$ and $h\in\bar{\mathcal{H}}$ with $\|h\|_{\bar{\mathcal{H}}}\leq M$.
\end{lem}

\begin{proof}

We consider the cases of $H>1/2$ and $H\leq 1/2$ separately. We only study the lower bound of $\Gamma_{\Phi_1(x;h)}$ since the upper bound is standard from pathwise estimates by \eqref{eq: expression Malliavin derivative} and~\eqref{Malliavin matrix X(h)}, plus the fact that $\|h\|_{\mathcal{H}}\leq M$.

\vspace{2mm}
\noindent
\emph{(i) Proof of the lower bound when  $H>1/2$.}
According to relation \eqref{Malliavin matrix X(h)} and the expression for the inner product in $\mathcal{H}$ given by \cite[equation (5.6)]{Nualart06}, we have \[
\Gamma_{\Phi_{1}(x;h)}=C_{H}\sum_{\alpha=1}^{d}\int_{[0,1]^{2}}J_{1}J_{s}^{-1}V_{\alpha}(\Phi_{s})V_{\alpha}^{*}(\Phi_{t})(J_{t}^{-1})^{*}J_{1}^{*}|t-s|^{2H-2}dsdt,
\]where we have omitted the dependence on $x$ and $h$ for $\Phi$ and $J$ inside the integral for notational simplicity. It follows that for any $z\in\mathbb{R}^N$, we have
\begin{align}\label{eq: inner product H>1/2}z^{*}\Gamma_{\Phi_{1}(x;h)}z=C_{H}\int_{[0,1]^{2}}\langle\xi_{s},\xi_{t}\rangle_{\mathbb{R}^{d}}|t-s|^{2H-2}dsdt,
\end{align}where
$\xi$ is the function in $\mathcal{H}$ defined by
\begin{align}\label{definition xi}\xi_{t}\triangleq V^{*}(\Phi_{t})(J_{t}^{-1})^{*}J_{1}^{*}z.\end{align}
According to an interpolation inequality proved by Baudoin-Hairer (cf. \cite[Proof of Lemma 4.4]{BH07}), given $\gamma>H-1/2$, we have 
\begin{equation}\label{eq: interpolation inequality}
\int_{[0,1]^{2}}\langle f_{s},f_{t}\rangle_{\mathbb{R}^{d}}|t-s|^{2H-2}dsdt\geq C_{\gamma}\frac{\left(\int_{0}^{1}v^{\gamma}(1-v)^{\gamma}|f_{v}|^{2}dv\right)^{2}}{\|f\|_{\gamma}^{2}}
\end{equation}for all $f\in C^\gamma([0,1];\mathbb{R}^d)$. 
Observe that, due to our uniform ellipticity assumption \eqref{eq:hyp-elliptic} and the non-degeneracy of $J_t$, we have
\begin{align}\label{lower bound xi}
\inf_{0\leq t\leq1}|\xi_{t}|^{2}\geq C_{H,V,M}|z|^{2}.
\end{align}
Furthermore, recall that $\Phi_t$  is driven by $h\in\bar{\mathcal{H}}$. We have also seen that $\bar{\mathcal{H}}\hookrightarrow C^H_0$ whenever $H>1/2$. Thus for $H-1/2<\gamma<H$, we get $\|\Phi_t\|_\gamma\leq C_{H,V}\|h\|_\gamma$; and the same inequality holds true for the Jacobian $J$ in \eqref{definition xi}. Therefore, going back to equation \eqref{definition xi} again, we have
\begin{align}\label{upper bound xi}\|\xi\|_{\gamma}^{2}\leq C_{H,V,M}\|h\|_{\bar{\mathcal{H}}}\,|z|^2\leq C_{H,V,M}|z|^{2},
\end{align}
where the last inequality stems from our assumption $\|h\|_{\bar{\mathcal{H}}}\leq M$. Therefore, taking $f_t=\xi_t$ in (\ref{eq: interpolation inequality}), plugging inequalities \eqref{lower bound xi} and \eqref{upper bound xi} and recalling inequality \eqref{eq: inner product H>1/2}, we conclude that \[
z^{*}\Gamma_{\Phi_{1}(x;h)}z\geq C_{H,V,M}|z|^{2}
\]uniformly for $\|h\|_{\bar{\mathcal{H}}}\leq M$ and the result follows.

\vspace{2mm}
\noindent
\emph{(ii) Proof of the lower bound when  $H\leq1/2$.}
Recall again that \eqref{Malliavin matrix X(h)} yields
$$z^{*}\Gamma_{\Phi_{1}(x;h)}z  =\|z^{*}D\Phi_{1}(x;h)\|_{\mathcal{H}}^{2}.$$
Then owing to the continuous embedding $\mathcal{H}\subseteq L^2([0,1])$ proved in Lemma \ref{lem: continuous embedding when H<1/2}, and expression~\eqref{eq: expression Malliavin derivative} for $D\Phi_t$, we have for any $z\in\mathbb{R}^N$,
\begin{align*}
z^{*}\Gamma_{\Phi_{1}(x;h)}z  & \geq C_{H}\|z^{*}D\Phi_{1}(x;h)\|_{L^{2}([0,1])}^{2}\\
 & =C_{H}\int_{0}^{1}z^{*}J_{1}J_{t}^{-1}V(\Phi_{t})V^{*}(\Phi_{t})(J_{t}^{-1})^{*}J_{1}^{*}zdt.
 \end{align*}
We can  now invoke the uniform ellipticity assumption \eqref{eq:hyp-elliptic} and the non-degeneracy of $J_t$ in order to obtain
 \begin{align*}
z^{*}\Gamma_{\Phi_{1}(x;h)}z \geq C_{H,V,M}|z|^{2}
\end{align*}uniformly for $\|h\|_{\bar{\mathcal{H}}}\leq M$. Our claim \eqref{uniform bound for Malliavin matrix} now follows as in the case $H>1/2$.
\end{proof}

With the preliminary results of Proposition \ref{th: summary of 4} and Lemma \ref{lem: uniform nondegeneracy of Malliavin matrix} in hand,  we are now able to complete the proof of Theorem \ref{thm: local lower estimate in elliptic case}.

\begin{proof}[Proof of Theorem \ref{thm: local lower estimate in elliptic case}] 

Recall that $X^\ep(h)$ is defined by \eqref{def approx X(h)}. According to our preliminary bound (\ref{eq: CM theorem}), it remains to show that 
\begin{equation}\label{eq: uniform lower estimate for Y^epsilon}
\mathbb{E}\left[\delta_{0}\left(X^\varepsilon(h)\right){\rm e}^{-I\left(\frac{h}{\varepsilon}\right)}\right]\geq C_{H,V}
\end{equation}uniformly in $h$ for $\|h\|_{\bar{\mathcal{H}}}\leq 2\varepsilon$ when $\varepsilon$ is small enough. The proof of this fact consists of the following two steps:

\vspace{2mm}
\noindent (i)  Prove that $\mathbb{E}[\delta_{0}(X(h)){\rm e}^{-I\left({h}/{\varepsilon}\right)}]\geq C_{H,V}$ for all $\varepsilon>0$ and $h\in\bar{\mathcal{H}}$ with $\|h\|_{\bar{\mathcal{H}}}\leq 1$;
\\
(ii) Upper bound the difference \[
\mathbb{E}\left[\delta_{0}\left(X^\varepsilon(h)\right){\rm e}^{-I\left(\frac{h}{\varepsilon}\right)}\right]-\mathbb{E}\left[\delta_{0}(X(h)){\rm e}^{-I\left(\frac{h}{\varepsilon}\right)}\right],
\]and show that it is small uniformly in $h$ for $\|h\|_{\bar{\mathcal{H}}}\leq 2\varepsilon $ when $\varepsilon$ is small.
We now treat the above two parts separately.

\noindent \emph{Proof of item (i):}  Recall that the first chaos $\mathcal{C}_1$ has been defined in Section 2.1. then observe that the random variable $X(h)=(X^1(h),...,X^N(h))$ introduced in Proposition \ref{th: summary of 4} sits in $\mathcal{C}_1$. We decompose the Wiener integral $I(h/\ep)$ as  $$I\left(h/\varepsilon\right)=G_{1}^{\varepsilon}+G_{2}^{\varepsilon},$$ where $G^\varepsilon_1$ and $G^\ep_2$ satisfy
$$G_1^\ep\in\mathrm{Span}\{X^i(h); 1\leq i\leq N\},\quad G^\ep_2\in \mathrm{Span}\{X^i(h); 1\leq i\leq N\}^\bot $$
where the orthogonal complement is considered in $\mathcal{C}_1$. With this decomposition in hand, we get
\begin{align*}
  \mathbb{E}\left[\delta_{0}(X(h)){\rm e}^{-I\left(\frac{h}{\varepsilon}\right)}\right]
 =\mathbb{E}\left[\delta_{0}(X(h)){\rm e}^{-G_{1}^{\varepsilon}}\right]\cdot\mathbb{E}\left[{\rm e}^{-G_{2}^{\varepsilon}}\right].
\end{align*}
Furthermore, $\mathbb{E}[\mathrm{e}^{G}]\geq 1$ for any centered Gaussian random variable $G$. Thus
\begin{align}\label{mid step-lower bound density}
  \mathbb{E}\left[\delta_{0}(X(h)){\rm e}^{-I\left(\frac{h}{\varepsilon}\right)}\right]
 \geq\mathbb{E}\left[\delta_{0}(X(h)){\rm e}^{-G_{1}^{\varepsilon}}\right].
\end{align} 
Next we approximate $\delta_0$ above by a sequence of function $\{\psi_n; n\geq 1\}$ compactly supported in $B(0,1/n)\subset\R^N$. Taking limits in the right hand-side of \eqref{mid step-lower bound density} and recalling that $G^\ep_1\in\mathrm{Span}\{X^i(h); 1\leq i\leq N\}$, we get
 \begin{align*}
  \mathbb{E}\left[\delta_{0}(X(h)){\rm e}^{-I\left(\frac{h}{\varepsilon}\right)}\right] \geq\mathbb{E}[\delta_{0}(X(h))].
\end{align*} 
We now resort to the fact that $X(h)$ is a Gaussian random variable with covariance matrix $\Gamma_{\Phi_1(x;h)}$ by \eqref{Malliavin matrix X(h)}, which satisfies relation \eqref{uniform bound for Malliavin matrix}. This yields
 \begin{align*}
  \mathbb{E}\left[\delta_{0}(X(h)){\rm e}^{-I\left(\frac{h}{\varepsilon}\right)}\right] \geq \frac{1}{(2\pi)^{\frac{N}{2}}\sqrt{\det \Gamma_{\Phi_1(x;h)}}}\geq C_{H,V},
\end{align*} 
uniformly for $\|h\|_{\bar{\mathcal{H}}}\leq1$. This ends the proof of item (i).

\medskip
\noindent\emph{Proof of item (ii):} By using the integration by parts formula in Malliavin's calculus (see e.g., \cite[Proposition 2.1.4]{Nualart06}, we have
$$\mathbb{E}[\delta_0(X(h))\mathrm{e}^{-I(h/\varepsilon)}]=\mathbb{E}\left[1_{\{X(h)\geq0\}}H(X(h), I({h}/{\varepsilon}))\right],$$ 
where $X(h)\geq 0$ is interpreted component-wise, and  $H(X(h), I({h}/{\varepsilon}))$ is a random variable which can be expressed explicitly in terms of the Malliavin derivatives of $I(h/\varepsilon)$, $X(h)$ and the inverse Malliavin covariance matrix $M_{X(h)}$ of $X(h)$. Similarly, we have
$$\mathbb{E}[\delta_0(X^\varepsilon(h))\mathrm{e}^{-I(h/\varepsilon)}]=\mathbb{E}\left[1_{\{X^\varepsilon(h)\geq0\}}H(X^\varepsilon(h), I({h}/{\varepsilon}))\right].$$ 
Therefore, 
\begin{align}\label{item ii breakdown}
&\left|\mathbb{E}[\delta_0(X(h))\mathrm{e}^{-I(h/\varepsilon)}]
-\mathbb{E}[\delta_0(X^\varepsilon(h))\mathrm{e}^{-I(h/\varepsilon)}]\right|\\
\leq&\left|\mathbb{E}\left[\left(\1_{\{X^\varepsilon(h)\geq0\}}-\1_{\{X(h)\geq0]\}}\right)H(X(h), I(h/\varepsilon))\right]\right|\nonumber\\
&+\left|\mathbb{E}\left[\1_{\{X^\varepsilon(h)\geq0\}}\left(H(X^\varepsilon(h),I(h/\varepsilon))-H(X(h),I(h/\varepsilon))\right)\right]\right|.\nonumber
\end{align}
Note that since $\|h\|_{\bar{\mathcal{H}}}\leq 2\varepsilon$ the random variable $H(X(h), I(h/\varepsilon))$ has bounded $p$-th moment (uniform in $\varepsilon$) for all $p\geq1$. It is thus clear from Proposition \ref{th: summary of 4}-(i) that the first term in the right-hand side of \eqref{item ii breakdown} can be made small when $\varepsilon$ is small.

As for the second term in the right-hand side of \eqref{item ii breakdown}, first note from standard argument (indeed, similar to the argument in the proof of Lemma \ref{lem: uniform nondegeneracy of Malliavin matrix}), one can show that $\det M_{X^\varepsilon(h)}$ has negative moments of all orders uniformly for all $\varepsilon\in (0,1)$ and bounded $h\in\bar{\mathcal{H}}$. Together with the convergence in Proposition \ref{th: summary of 4}-(i), we can show that 
\begin{align}\label{convergence Malliavin matrix}
\det M_{X^{\varepsilon}(h)}^{-1}\stackrel{L^{p}}{\longrightarrow}\det M_{\Phi_{1}(x;h)}^{-1},\qquad {\rm as}\ \varepsilon\rightarrow0 ,
\end{align}uniformly for $\|h\|_{\bar{\mathcal{H}}}\leq 1$ for each $p\geq1.$ Now recall that  $H(X(h), I({h}/{\varepsilon}))$ is a random variable which can be expressed explicitly in terms of the Malliavin derivatives of $I(h/\varepsilon)$, $X(h)$ and the inverse Malliavin covariance matrix $M_{X(h)}$ of $X(h)$. The convergence in \eqref{convergence Malliavin matrix} and Proposition \ref{th: summary of 4}-(i) is sufficient to conclude that the second term in the right-hand side of \eqref{item ii breakdown} can be made small when $\varepsilon$ is small. Therefore, the assertion of item (ii) holds.

Once item (i) and (ii) are proved, it is easy to obtain \eqref{eq: uniform lower estimate for Y^epsilon} and the details are omitted. This finishes te proof of Theorem \ref{thm: local lower estimate in elliptic case}. 
\end{proof}

\bigskip
We conclude our discussion by a remark regarding SDE with a drift.
\begin{rem}
One can also consider the SDE in \eqref{sde-intro} but with a smooth and bounded drift
\begin{align}\label{sde-drift}
Z_t=x+\int_0^tV_0(Z_s)ds+\sum_{i=1}^d \int_0^t V_i(Z_s)dB^i_s,\quad t\in[0,1].
\end{align}
It turns our the control distance of the system \eqref{sde-drift} (in terms of large deviation etc.) is the same as the one without a drift; that is, the same as being defined in \eqref{def: distance}. Hence, the corresponding local lower bound for the density function of $Z_t$ is the same as stated in Theorem \ref{eq:local-density-bound}. In order to see this, recall that  $\Phi_t(x;\cdot): \bar{\mathcal{H}}\to C[0,1]$ is the deterministic It\^{o} map associated to equation \eqref{sde-intro}. For each $\epsilon>0$ we further define $\Phi^\epsilon_t(x;\cdot)$ to be the solution map of the equation
\begin{align*}
Z_t^\epsilon=x+\epsilon^{\frac{1}{H}}\int_0^tV_0(Z_s^\epsilon)ds+\sum_{i=1}^d\int_0^tV_i(Z_s^\epsilon)dB^i_s,\quad t\in[0,1].
\end{align*}
That is, $Z^\epsilon_t=\Phi^\epsilon_t(x; B)$. Similar to \eqref{scaling property of equation driven by fBm}, we have for $\epsilon=t^H$,
\begin{align*}
Z_t=\Phi_{t}^1(x;B)\stackrel{{\rm law}}{=}\Phi_{1}^\epsilon(x;\varepsilon B).
\end{align*}
Now we proceed as in the proof of Theorem \ref{eq:local-density-bound}, and denote by $p(t,x,z)$ the density function of $Z_t$. Equation \eqref{b3} becomes
$$p(t,x,z)=\mathbb{E}\left[\delta_{z}\left(\Phi_{1}^\epsilon(x;\varepsilon B)\right)\right]=\mathbb{E}\left[\delta_{z}\left(\Phi_1(x;\epsilon B)+(\Phi_{1}^\epsilon(x;\varepsilon B)-\Phi_1(x;\epsilon B))\right)\right].$$
\end{rem}
As a result, if we still pick $h\in \Pi_{x,z}$ as before (that is, $\Phi_1(x,h)=z$), the expectation on the right hand-side of \eqref{eq: CM theorem} becomes
$$\mathbb{E}\left[\delta_{0}\left(X^\ep(h)+\frac{\Phi_1^\epsilon(x;\varepsilon B+h)-\Phi_1(x,\varepsilon B+h)}{\epsilon}\right){\rm e}^{-I\left(\frac{h}{\varepsilon}\right)}\right].$$
The observation is that rough differential equations are Lipschitz continuous with respect to the vector fields. Hence  the extra term
$$\frac{\Phi_1^\epsilon(x;\varepsilon B+h)-\Phi_1(x,\varepsilon B+h)}{\epsilon}$$
is of order $\epsilon^{\frac{1}{H}-1}$, and can be considered negligible since $0<H<1$. Therefore, all the previous argument goes through as if there was no drift. We leave it to the enterprising readers to fill in the details.

\end{document}